\documentclass[12pt,a4paper]{article}
\usepackage{amssymb, latexsym, amsmath, amsthm}
\usepackage{amsfonts,amsmath,amsthm,amssymb}
\usepackage[english]{babel}
\usepackage{a4wide}
\usepackage{color}
\usepackage{amsfonts}
\newtheorem{theorem}{Theorem}[section]
\newtheorem{prop}[theorem]{Proposition}
\newtheorem{lemma}[theorem]{Lemma}
\newtheorem{corollary}[theorem]{Corollary}
\theoremstyle{definition}
\newtheorem{definition}[theorem]{Definition}
\newtheorem{example}[theorem]{Example}
\theoremstyle{remark}
\newtheorem{remark}[theorem]{Remark}
\theoremstyle{remarks}

\newtheorem{Problem}{\bf Question}
 \numberwithin{equation}{section}

\newcommand{\si}{\sigma}

\newcommand{\an}[1]{\mbox{\rm lann}_R({#1})}
\newcommand{\ran}[1]{\mbox{\rm rann}_R({#1})}

\begin{document}

\title{ Ring  Endomorphisms with Large Images}
\author{ {\bf \normalsize Andr\'{e} Leroy and Jerzy
Matczuk}\footnote{ The research was
 supported by Polish MNiSW grant No. N N201 268435}
\vspace{6pt}\\
 \normalsize  Universit\'{e} d'Artois,  Facult\'{e} Jean Perrin\\
\normalsize Rue Jean Souvraz  62 307 Lens, France\\
   \normalsize  E-mail: leroy@euler.univ-artois.fr \vspace{6pt}\\
 \normalsize $^* $Institute of Mathematics, Warsaw University,\\
\normalsize Banacha 2, 02-097 Warsaw, Poland\\
 \normalsize E-mail: jmatczuk@mimuw.edu.pl}
\date{ }
\maketitle\markboth{\rm  A. LEROY AND J.MATCZUK}{ \rm Endomorphisms with Large Images}

\begin{abstract}
%We initiate the study of endomorphisms having large images.
%
%Other possibility:
The notion of ring endomorphisms having large images is introduced. Among others,
injectivity and surjectivity of such endomorphisms
 are studied.  It is proved, in particular,  that an endomorphism $\si$ of a prime one-sided noetherian ring $R$ is injective whenever  the image $\si(R)$ contains an essential left ideal $L$ of $R$. If additionally $\si(L)=L$, then $\si$ is an automorphism of $R$. Examples showing that the assumptions  imposed on $R$ can not be weakened to $R$ being a prime left Goldie ring are provided. Two open questions are formulated.
\end{abstract}

%\section{Introduction}
In this paper we start investigations of endomorphisms of semiprime
unital rings $R$ having large images, i.e. endomorphisms $\si$ such that the image
$\si(R)$ contains an essential left ideal of $R$
 (see Definition \ref{def of large image}). The
 motivation for such studies is twofold.
Let us recall that a ring (or a module) is called Hopfian (resp. co-Hopfian) if every
surjective (resp. injective) endomorphism is injective (resp.
surjective).   It is well known and easy to prove  that noetherian (artinian) modules
and rings are Hopfian (co-Hopfian). However,  in general,    the Hopfian property for
modules behaves much better than that of rings.     Examples showing a difference in that
behaviour can be found in    \cite{H}, \cite{T}, \cite{V2}, \cite{V3}.  In case of rings, the set of all endomorphisms has no natural structure of a ring and it seems to be natural to consider some classes of endomorphisms of a ring. Our goal is to investigate how one can weaken Hopfian or co-Hopfian  assumptions on a ring endomorphism to conclude that the endomorphism is  injective or surjective. We obtained some positive results in this direction (Cf. the second part of the introduction). Surprisingly we could not answer some of the elementary formulated problems. For example we proved that every endomorphism having large image of a prime ring with Krull dimension has to be injective, however we do not know whether the same property holds for semiprime rings with Krull dimension.

The second source of motivation  for our studies
 is   lifting of properties from a nonzero ideal of a prime ring to the ring itself.
Theorems 1.1  of \cite{BFL}, 1, 2 and 3 of \cite{B},   Main Theorem of \cite{BMM}, Th\`{e}or\'{e}me 1.9 of \cite{LM} and Chapters 5 and 6 of the book \cite{B1} can serve as examples of results of such nature.

The paper is organized as follows: Section 1 deals with the injectivity of endomorphisms of semiprime left noetherian rings having large images.  Necessary and sufficient conditions for such endomorphisms of a semiprime ring  to be injective are given in Proposition \ref{prop. injectivity for semiprime noeterian}. Theorem \ref{thm prime injective} says that such an endomorphism is always injective provided $R$ is prime.
In Section 2 we investigate surjectivity of these endomorphisms. In particular we show in Corollary
\ref{cor. endo implies auto}, that every endomorphism of a semiprime left noetherian ring such that $\si(L)=L$, for an essential left ideal $L$ of $R$ has to be an automorphism. We show also that every endomorphism of a principal left ideal domain having a large image  is always an automorphism. Examples are provided all along the paper to justify the assumptions made.
As an application of our results, we obtain that the Jacobian conjecture has a positive solution for endomorphisms with large images.
 Finally two open questions are formulated.

\section{Injectivity}
Throughout the paper $\si$ will stand
for an endomorphism of   an associative ring $R$ with unity.

The left annihilator $\{a\in R\mid aS=0\}$ of a subset $S$ of $R$ will be denoted by
$\an{S}$. The right annihilator of $S$ in $R$ will be denoted by $\ran{S}$.

The following proposition is a part of folklore. It gives some basic motivation for the assumptions we work with. We left its easy proof to the reader.

\begin{prop}\label{folklore}
Let $L$ be a left ideal of a ring $R$ such that $\an{L}=0$.
 Suppose that $\si$ is an  endomorphism  and $\tau$ is an automorphism  of $R$ such that $\si|_L=\tau|_L$.
 Then $\si=\tau$ is an automorphism of $R$.
\end{prop}

 The assumption of Proposition \ref{folklore} is satisfied if $L$ is an essential left
 ideal of a   semiprime ring. Thus, in particular, when   $L$ is any nonzero
ideal  of a prime ring.

\begin{lemma}
\label{injective for L fixed by sigma} Let $\si $ be  an endomorphism    of a ring $R$ and $n\geq 1$ a natural number. Then
 $\ker \si^n \subset
\ker\si^{n+1}$ iff
 $\si^n(R)\cap\ker\si\ne 0$
iff $\si^n(R)\cap\ker\si^n\ne 0$.
\end{lemma}
\begin{proof} The easy proof of the lemma is left  as an exercise.
 \end{proof}
It is known that   rings satisfying the ACC condition on ideals  are Hopfian. A direct application of the above lemma offers a  generalization of this fact.
\begin{prop} \label{ cor injective for L fixed by sigma} Let  $R$ be a ring satisfying the ACC condition on ideals. Suppose that, for any $n\geq 1$,  there exists a nonzero left ideal $L=L(n)$ of $R$ such that $L\subseteq \si^n(R)$ (for example when $\si(L)=L$).
If either $L$ is essential as a left ideal of $R$ or $\an{L}=0$, then
 $\si$ is injective.
\end{prop}

The following examples justify the assumptions made in the above proposition and will help delimiting
the ones that will appear in Theorem \ref{thm prime injective}.

\begin{example}
Let $K$ be a field. The $K$-endomorphism of the polynomial ring $K[x]$ which sends $x$ onto $x^2$ induces an endomorphism $\si$ of 
the ring $R=K[ x\mid x^3=0]$. Then the image $\si(R)$ contains the essential ideal $Rx^2=Kx^2$,  $\ker \si\ne 0$  and $R$ is 
noetherian.

\end{example}

\begin{example}\label{example nonzero ker}
 Let $R=K[x_i\mid i\ge 0]$ be a polynomial ring in indeterminates $x_i$, $i\geq 0$. and $\si$ be the $K$-endomorphism of $R$ given by $\si (x_0)=x_0, \; \si(x_1)=0$ and $\si(x_i)=x_{i-1}$, for $i \ge 2$. Then $R$ is a domain, $\si(R)=R$ so $\si(R)$ contains a nonzero ideal and $\si$  is not injective.
\end{example}

\begin{example}
\label{example semiprime noetherian}
Let $R=K[x,y\mid xy=yx=0]$ where $K$ is a field.  Let $\sigma$ be the $K$-linear endomorphism
of $R$ determined by $\sigma(x)=x$ and $\sigma(y)=0$.  $R$ is semiprime noetherian and
$(x)\subseteq \sigma (R)$ but $\sigma$ is not injective.
\end{example}

Let us observe that for a left ideal $L$ of a ring $R$ the properties of being essential and having zero left annihilator are independent notions, however they coincide   when $R$ is a semiprime left Goldie ring.

We will see in Theorem \ref{thm prime injective}  that for prime rings   a much stronger statement than the one given in the above Proposition \ref{ cor injective for L fixed by sigma} holds. To get this, some preparation is needed.

\begin{definition}\label{def of large image}
 We say that an endomorphism $\si$ of a  ring $R$ has a  large image if $\si(R)$ contains an essential left ideal $L$ of $R$ with $\ran{L}=0$.
\end{definition}
Notice that the above definition is not left-right symmetric as   the following example shows:
\begin{example}
 Let $K$ be a field and $R=K\langle x_i\mid i\geq 0\; \mbox{\rm and}\; x_kx_l=0 \;\mbox{when}\; k\geq l\rangle$. The ideal $M$ generated by the set $\{x_i\}_{i=0}^\infty$ is nil, so $R$ is a local ring. Notice also
 that $Mx_0=0$ while $\an{M}=0$.

 It is easy to check that the assignment $\si(x_0)=0$ and $\si(x_{i+1})=x_i$ defines a $K$-linear endomorphism of $R$. If $L$ is a left ideal of $R$, then $L$ is  contained in $M$. Thus $L$   has nonzero right annihilator. This means that  $\si$ does not satisfy Definition \ref{def of large image}. However $M\subseteq \si(R)$ is essential as a right ideal of $R$ and $\an{M}=0$.
\end{example}
When $L$ is an essential left ideal of a   ring $R$ then $\ran{L}=0$  if the left singular ideal of $R$ is zero. This is always the case when  $R$ is semiprime left Goldie. In particular, when $R$  is a  semiprime ring which is left noetherian, or   possesses left Krull dimension, then any essential left ideal $L$ of $R$ has a zero right annihilator.

In what follows, $\mathcal K(_RM)$  and $\mbox{\rm udim}(_RM)$ will denote the Krull dimension and the Goldie dimension of a left $R$-module $M$.
 Let us recall that any left noetherian ring has  left Krull dimension $\mathcal
 K(R)=\mathcal K(_RR)$ and left Goldie dimension $\mbox{\rm udim}(R)$.

Notice that if $\si$ is an endomorphism of a semiprime ring $R$ having a large image, then the rings $R$ and $\si(R)$ share many ring properties. Some of them are recorded in the following:
\begin{prop}\label{properties of large endomorhisms}  Suppose that an endomorphism  $\si$ of a ring $R$ has a large image. Then:
\begin{enumerate}
  \item[1.]  If $R$ is prime (semiprime), then so is $\si(R)$;
  \item[2.] $\mbox{\rm udim}(\si(R))=\mbox{\rm udim}(R)$.
  \end{enumerate}
    If  additionally   $R$ is semiprime, then:
    \begin{enumerate}
\item[3.] If  $R$ is a left Goldie ring, then $\si(R)$ is  also a semiprime left Goldie ring and  the classical left quotient rings $Q(\si(R))$ and $Q(R)$ are equal.

    \item[4.]  If $\mathcal K(R)$ exists, then $\mathcal K(\si(R))=\mathcal K(R)$.
  \item[5.] If $R$ is left noetherian, then $R$ is also noetherian as a left module over $\si(R)$.
   \end{enumerate}
\begin{proof} Statements (1), (2) and (3), which probably are a part of folklore, hold  in a more general context when $\si(R)$ is replaced by an arbitrary subring $T$ of $R$ containing an essential   left ideal $L$ of $R$ with $\ran{L}=0$. The first one, does not require essentiality of $L$ and   is an easy consequence  of the following observation. Let $a,b\in T$ be such that $aLb=0$. Then $LaRLb=0$ and  $La=0$ only if $a=0$, as $\ran{L}=0$.

  If $V$ is a nonzero left ideal of $T$, then $LV\subseteq V$  is a left ideal of $R$ contained in $T$ and $LV\ne 0$ as $\ran{L}=0$.  Let $L_1+L_2+\ldots $ be a direct sum of nonzero left ideals of $T$. Then, by the above,    $LL_1+LL_2\ldots  $ is a direct sum of nonzero left ideals of $R$. Notice also  that if   $M_1+M_2+\ldots $ is  a direct sum  of nonzero left ideals of $R$, then   $(L\cap M_1)+(L\cap M_2)+\ldots $ is a direct sum of nonzero ideals of $T$. This implies that $\mbox{\rm udim}(T)=\mbox{\rm udim}(R)$, i.e. (2) holds.

  Suppose $R$ is semiprime left Goldie. Then, by the above,  $T$ is semiprime and has finite left Goldie dimension.  This yields that $T$ is a left Goldie ring, as the a.c.c.  on left annihilators is inherited by subrings. Since $L$ is an essential left ideal of a semiprime Goldie ring $R$, there exists a regular element
$c$ of $R$ such that $Rc\subseteq L\subseteq T$.  Clearly $c$ is regular in $T$, so we have $R=Rcc^{-1}\subseteq Q(T)$.  Notice also that if $d\in R$ is a
regular element of $R$ then $dc\in T$ is a regular element of $T$. Therefore  $d$ is an
invertible element of $Q(T)$. This shows that $T\subseteq R\subseteq Q(R) \subseteq Q(T)$.
On the other hand, the essentiality of $L$ in $R$ shows that   regular elements of $T$
are  left regular in $R$ and, consequently, regular in $R$, as $R$ is semiprime left Goldie (see Proposition 2.3.4. of \cite{MC}).  This implies that $Q(T)\subseteq Q(R)$ and completes the proof of $(3)$.

(4) Suppose $\mathcal{K}(R)$ exists. Then
 $\mathcal{K}(R)\geq \mathcal{K}(\si(R))$ (Cf. Lemma 6.3.3\cite{MC}).
  The endomorphism $\si$ has a large image, so $\si(R)$ contains an essential left ideal $L$ of $R$ such that $\ran{L}=0$. By assumption, $R$ is a semiprime ring with Krull dimension. Thus, by Proposition 6.3.10(ii) of \cite{MC}, $\mathcal{K}(R)=\mathcal{K}(_RL)$.
The statements $(1)$ and $(2)$ above show that $\si(R)$ is a semiprime ring and $L$ is also an essential left ideal of $\si(R)$. Thus, by the same proposition we also have $\mathcal{K}(\si(R))=\mathcal{K}(_{\si(R)}L)$.
This yields $\mathcal{K}(R)\leq \mathcal{K}(\si(R))$, as $\mathcal{K}(_RL)\leq \mathcal{K}(_{\si(R)}L)$ and completes the proof of (3).

(5) Suppose $R$ is left noetherian and $L$ is an essential left ideal of $R$ contained in $\si(R)$. Then $R$ is a semiprime left Goldie ring, so $L$ contains a
regular element of $R$, say $c\in L$ is such. Then $Rc\subseteq L\subseteq \si(R)$
is a submodule of a noetherian left $\si(R)$-module $\si(R)$.
 Thus we can find $c_i\in \si(R)$ such that $Rc=\sum_{i=1}^n\;\si(R)c_i$.
Since $c_i\in Rc$, $c_i=r_ic$ for some $r_i\in R$. Now
$Rc=\sum_{i=1}^n\;\si(R)r_ic$ and $R=\sum_{i=1}^n\;\si(R)r_i$
follows, as $c$ is regular in $R$.
\end{proof}
\end{prop}
Notice that all  the above statements do  not imply that $\si(R)$ contains an essential left ideal of $R$. Indeed, if $R=K[x]$ is a polynomial ring over a field $K$ and $\si$ is an $K$-endomorphism of $R$ defined by $\si(x)=x^2$. Then clearly $R$ and $\si(R)$ possess all properties from the above proposition but $\si(R)$ does not contain a nonzero ideal of $R$.

 As an immediate application of Proposition \ref{properties of large endomorhisms} we obtain the following:
\begin{theorem}\label{thm prime injective} Let $\si$ be an endomorphism  of a semiprime ring $R$ which has a large image.
  Suppose $R$ has  left Krull dimension (for example $R$ is left noetherian), then $\ker\si$ is not  essential as  a left ideal of $R$. In particular, if additionally $R$ is a prime ring, then $\si$ is an injective endomorphism.
  \end{theorem}
\begin{proof}
 Suppose $\mathcal{K}(R)$ exists. Thus, by Proposition \ref{properties of large endomorhisms}, $\mathcal{K}(\si(R))=\mathcal{K}
 (R)$. Assume that $\ker \si$ is essential as a left ideal.  Then, as   $R$ is a semiprime left Goldie ring, there exists a regular 
 element $c$ of $R$ such that  $c\in \ker \si$. Hence, by Lemma 6.3.9 \cite{MC}, $\mathcal{K}(_R(R/Rc))< \mathcal{K}(R)$.  Moreover 
$\mathcal{K}(_R(R/\ker\si))\leq\mathcal{K}(_R(R/Rc)) $ as  $R/\ker\si$ is a homomorphic image of $R/Rc$ as a left  $R$-module. This 
implies that $\mathcal{K}(\si(R))=\mathcal{K}(_R(R/\ker\si))<\mathcal{K}(R)$, which is impossible. This contradiction shows that 
$\ker \si$ can not be essential as a left ideal of $R$. This completes the proof of the theorem.
\end{proof}

 We do not know the answer to the following:
 \begin{Problem}
  Suppose that an endomorphism $\si $ of a semiprime left noetherian ring has a large image. Does $\si$ have to be injective?
 \end{Problem}
 The next proposition offers some equivalent
conditions for  an endomorphism of  a semiprime left noetherian ring to be injective. In order to get this  some
preparation is needed.
Let $\mathcal{C}(R)$ denote the set of all regular elements of a ring $R$ and
$\mathcal{C}_l(R)$ stand for the set of all left regular elements of $R$, i.e. $\mathcal{C}_l(R)=\{a\in R\mid \an{a}=0\}$. When $R$ is a semiprime left  Goldie ring
then $\mathcal{C}(R)=\mathcal{C}_l(R)$ (Cf. Proposition 2.3.4\cite{MC}).  Jategoankar
proved in \cite{Ja} that $\si(\mathcal{C}(R))\subseteq \mathcal{C}(R)$, for any injective  endomorphism $\si$ of a semiprime left Goldie ring $R$.
The following elementary lemma offers the same thesis under a slightly different hypothesis.
\begin{lemma}\label{lemma on regular elements}
Let $\si$ be an injective endomorphism of a ring   $R$ having large image.  Then $\si(\mathcal{C}_l(R))\subseteq \mathcal{C}_l(R)$. In particular, $\si(\mathcal{C}(R))\subseteq \mathcal{C}(R)$, provided $R$ is semiprime left Goldie.

 \end{lemma}
  \begin{proof} Since $\si$ has a large image, $\si(R)$ contains an
 essential left ideal $L$ of $R$.  Let $a\in R$. Suppose that $N=\an{\si(a)}\ne 0$. Then $0\ne N\cap L\subseteq \si(R)$.
  This means that there is a nonzero    $  r\in R$ such that  $\si(r)\si(a)=0$. Hence $ra=0$, as $\si$ is injective. This shows that $\si(\mathcal{C}_l(R))\subseteq \mathcal{C}_l(R)$.
 Suppose $R$ is semiprime left Goldie. Then  $\mathcal{C}(R)=\mathcal{C}_l(R)$ and the thesis follows.
   \end{proof}
  \begin{theorem}\label{prop. ext. to Q(R)}
   Let $R$ be a semiprime left Goldie ring with an endomorphism $\si$. Suppose  $\si$ has  a large image. Then the following conditions are equivalent:
   \begin{enumerate}
     \item  $\si$ is injective;
     \item  $\si(\mathcal{C}(R))\subseteq \mathcal{C}(R)$.
   \end{enumerate}
     If one of the above conditions holds then  $\si$      extends, in a canonical   way,
    to an automorphism of
 the left classical quotient ring $Q(R)$ of $R$.
  \end{theorem}
   \begin{proof}
    The implication $(1)\Rightarrow (2)$ is given by Lemma \ref{lemma on regular elements}.

    Suppose that  $\si(\mathcal{C}(R))\subseteq \mathcal{C}(R)$. This implies that $\si$ can be uniquely extended  to an endomorphism, also denoted by $\si$,  of $Q=Q(R)$.   We claim that $\si$ is an automorphism of $Q$.  By the theorem of Goldie, $Q$ is a semisimple ring. This means that its homomorphic image $\si(Q)$ is also a semisimple ring. This and Proposition \ref{properties of large endomorhisms}(3) imply  that $\si(R)$ is a semiprime left Goldie ring and $Q(\si(R))\subseteq \si(Q)\subseteq Q=Q(\si(R))$. This shows that $\si(Q)=Q$ and  Proposition \ref{ cor injective for L fixed by sigma} implies that  $\si$ is also injective, since   $Q$ is left noetherian, as a semisimple ring. This completes the proof.
   \end{proof}

\begin{prop}\label{prop. injectivity for semiprime noeterian} Suppose $R$ is a semiprime left noetherian ring. Let $\si$ be an endomorphism  of $R$ having a large image.  The following conditions are equivalent:
\begin{enumerate}
  \item  $\si$ is injective;
  \item $\si(\mathcal{C}(R))\subseteq \mathcal{C}(R)$;
  \item there exists a regular element $c\in R$ such that $Rc\subseteq\si (R)$ and $\si^n(c)$ is regular in $R$, for every $n\in \mathbb{\mathbb{N}}$;
  \item  $\si^n$ has a large image, for every $n\in \mathbb{\mathbb{N}}$.
\end{enumerate}
\end{prop}
\begin{proof}
 The implication $(1)\Rightarrow (2)$ is given by Lemma \ref{lemma on regular elements}.

 By the assumption, $R$ is semiprime left Goldie and $\si(R)$ contains an essential ideal $L$ of $R$. Thus, there exists a regular element $c$ of $R$ such that $Rc\subseteq \si (R)$. It is clear now, that $(2)\Rightarrow (3)$.

 $(3)\Rightarrow (4)$.
 Let $c\in R$ be as described in (3) and set $L=Rc$.  Since $L$ is a left ideal of $R$ contained in $\si(R)$, $\si^k(L) $ is a left ideal of  $ \si^{k}(R)$ contained in $\si^{k+1}(R)$, for every $k\geq 0$.   We prove, by induction on $m\geq 0$, that $P_m=L\si(L)\ldots \si^m(L)\subseteq\si^{m+1}(R)$. Notice that, by assumption,
 $P_0=L\subseteq \si(R)$. Let $m\geq 0$ and assume that $P_m\subseteq \si^{m+1}(R)$. Then $P_{m+1}=P_m\si^{m+1}(L)\subseteq \si^{m+1}(L)\subseteq \si^{m+2}(R)$, as $P_m\subseteq \si^{m+1}(R)$ and $\si^{m+1}(L)$ is a left ideal of $\si^{m+1}(R)$ contained in $\si^{m+2}(R)$. This proves the claim. Notice that
 $P_m$ contains a regular element $c\si(c)\ldots\si^m(c)$, where $m\geq 0$. This implies that, for any $n\geq 1$, $\si^n$ has a  large image  i.e. (4) holds.

 $(4)\Rightarrow (1)$.
  $R$ is left noetherian and the implication is a direct consequence of Proposition \ref{ cor injective for L fixed by sigma}.
\end{proof}
Let us observe  that  implications $(1)\Rightarrow (2)\Rightarrow (3)\Rightarrow (4)$ are satisfied in the above proposition under the assumption that $R$ is a semiprime left Goldie ring. The assumption that  $R$ is left noetherian was  used  in the proof of $(4)\Rightarrow (1)$ only (but semiprimeness of $R$ was not used). In fact, Example \ref{example nonzero ker} shows that the implication $(3)\Rightarrow (1)$ does not hold when $R$ is not noetherian even if it  is a commutative domain.

\section{Surjectivity}
 In this section,  we will assume that the endomorphism $\si$ of the ring $R$ is injective. It
is known (Cf. \cite{Jo}, \cite{Ma}) that in this situation, there exists a universal overring
$A=A(R,\si)$ of $R$, called a Cohn-Jordan extension of $R$, such that $\si$
extends to an automorphism of $A$ and,  for any $a\in A$, there exists $n\ge 1$ such that $\si^n(a)\in R$. Notice that $\si$ is an automorphism of $R$ if and only if $R=A$.

The following technical lemma will be helpful in the proof of Theorem \ref{thm automorphism}.
\begin{lemma} \label{techcal jordan ext.} Let $\sigma$ be an injective endomorphism of a  ring $R$ and $L$ be a  left (two-sided) ideal of $R$. Then $L$ is a left (two-sided) ideal of $A=A(R,\si)$ iff
 $\si^n(L)$  is a left (two-sided) ideal   of $R$, for every $n\in \mathbb{N}$.
\end{lemma}
\begin{proof}
 Suppose that, for any $n\in \mathbb{N}$,  $\si^n(L)$  is a left  ideal   of $R$. Let
 $a\in A$ and $n\ge 1$ be such that  $\sigma^n(a)\in R$. Then, by assumption, $\si^n(aL)=\si^n(a)\si^n(L)\subseteq \si^n(L)$. Injectivity of  $\si$ implies    $aL\subseteq L$   and shows that $L$ is a left ideal of $A$. Similar arguments work  on the right side.

 The reverse  implication is clear as, for a left ideal $J$ of $A$, $\si(J)$ is a left ideal of $A$ and $\si(J)\subseteq R$, provided $J\subseteq R$.
\end{proof}
\begin{prop}\label{prop. injective implies auto}
Let $\sigma$ be an injective endomorphism of a left noetherian ring $R$. Then $\si$ is an automorphism of $R$ iff
 there exists  an element $c\in R\cap \mathcal{C}(A)$
 such that $Ac\subseteq R$, where $A=A(R,\si)$.
\end{prop}
\begin{proof} Suppose $c\in R$ is as in the proposition.
Let $a\in R$.  Notice that, as $Ac\subseteq R$, we have $R\sigma^{-m}(a)c\subseteq R$,   for every $m\ge 1$. This means that $I_m=\sum_{i=0}^mR\sigma^{-i}(a)c$ is a left ideal of $R$, for every $m\ge 1$. Since $R$ is left noetherian and $I_m\subseteq I_{m+1}$, for any $m$, there exists
$n\ge 1$ such that $\sigma^{-(n+1)}(a)c\in I_n
$.  This and regularity of  $c$   in $A$ imply that there are elements    $r_0,\dots,r_n\in R$ such that
$\sigma^{-(n+1)}(a)=\sum_{i=0}^nr_i\sigma^{-i}(a)$. Now,
applying $\sigma^{n+1}$ on
both sides of this equality we obtain  $a\in \sigma(R)$. This shows that $\si(R)=R$,    so $\si$ is an automorphism.

For the reverse implication it is enough to take $c=1$.
\end{proof}

Now we are in position to prove the following:
\begin{theorem} \label{thm automorphism} Let $\si$ be an injective endomorphism of a  left noetherian semiprime ring $R$.  Then the following conditions are equivalent:
\begin{enumerate}
  \item  $\si$ is an automorphism of $R$;
  \item There exists an essential left ideal $L$ of $R$ such that $L\subseteq \si(R)$ and $\si^n(L)$ is a left ideal of $R$, for every $n\in\mathbb{N}$.
\end{enumerate}

\end{theorem}
\begin{proof} It is enough to prove $(2)\Rightarrow (1)$. Let $L$ be an essential left ideal of $R$ such that $L\subseteq \si(R)$ and $\si^n(L)$ is a left ideal of $R$, for every $n\in\mathbb{N}$.  Lemma \ref{techcal jordan ext.} and assumptions imposed on $L$ yield that $L$ is a left ideal of $A=A(R;\si)$.    Since $L$ is an essential left ideal  of a semiprime left Goldie ring, it contains a regular element $c\in \mathcal{C}(R)$.  Injectivity of $\si$ and Lemma \ref{lemma on regular elements}  imply that, for any $n\geq 0$, $\si^n(c)\in\mathcal{C}(R)$.
         Let $a\in A$ and  $n\ge 1$ be such that $ac=0$ and $\si^n(a)\in R$. Then $\si^n(a)\si^n(c)=0$ and $a=0$ follows as $\si^n(c)\in \mathcal{C}(R)$ and $\si$ is injective. Similarly $ca=0$ implies $a=0$. This shows that $c\in\mathcal{C}(A)$.
We also have $Ac\subseteq AL\subseteq R$ and Proposition \ref{prop. injective implies auto} completes the proof.
\end{proof}

Let us remark  that, by  Theorem \ref{thm prime injective},  an endomorphism   $\si$ satisfying   Statement (2) of Theorem \ref{thm automorphism} is injective, provided $R$ is a prime ring.

By the above Theorem \ref{thm automorphism} and Proposition \ref{ cor injective for L fixed by sigma} gives the following:
\begin{corollary}\label{cor. endo implies auto}
Let $R$ be a semiprime  left noetherian ring and $L$ be  an essential left ideal of $R$. Then every endomorphism $\si$ of $R$ such that $\si(L)=L$ is an automorphism of $R$.
\end{corollary}
 Examples \ref{example nonzero ker} and \ref{local}  show that the noetherian assumption in the above corollary  is essential even in the case $R$ is a commutative domain. Namely, in general,
 an endomorphism $\si$ as in the above example does not have to be injective. There exist also such injective endomorphisms which are not onto.

 Every nonzero ideal of a prime ring is essential as a left ideal, thus using Theorem \ref{thm prime injective} and taking $L=\si(I)$ in Theorem \ref{thm automorphism}    we get the following:
\begin{corollary}
An endomorphism $\si$ of a prime left
 noetherian ring $R$ is an automorphism iff there exists an ideal $I$ of $R$ such that $\si(I)\ne 0$ and $\si^n(I)$ is an ideal of $R$, for every $n\geq 1$.
\end{corollary}

\begin{prop}\label{prop imagies of ideals}
Let $\si$ be an endomorphism of a prime left noetherian ring such that $\si(R)$ contains a nonzero ideal $I$ of $R$. Then, for any natural number $n\geq 1$ we have:
 \begin{enumerate}
   \item  $\si^n (R)$ contains a nonzero ideal $I_n$ of $R$ such that $I_{n+1}\subseteq \si(I_n)$;
   \item there exists a nonzero ideal $J$ of $R$ such that  $0\ne\si^i(J)$ is an ideal of $R$, for all $1\leq i\leq n$.
        \end{enumerate}
\end{prop}

\begin{proof}(1)  By assumption $0\ne I\subseteq \si(R)$ and Theorem \ref{thm prime injective} shows that $\si$ is injective.  We construct $I_n$ by induction as follows: $I_1=I$ and $I_{n+1}=I\si(I_n)I$, for $n\ge 1$.  The injectivity of $\si$ and the primeness of $R$ show that $I_{n+1}$ is a nonzero ideal of $R$.  Moreover,
making use of the induction hypothesis we have: $I_{n+1}=I\si(I_n)I\subseteq
\si (R)\si(I_n)\si(R)\subseteq \si(I_n)\subseteq \si^{n+1}(R)$.  This gives the proof of (1).

 (2)   By (1), there exists a nonzero ideal $I_n$ of $R$ contained in $\si^n(R)$. It is enough to take $J=\si^{-n}(I_n)$.
\end{proof}
If $n\geq 1$ and $J$, $J'$ are ideals of $R$ such that $\si^i(J)$ and $\si^i(J')$ are ideals of $R$, for $0\leq i\leq n$, then $J+J'$ also has this property. This means that, for any $n\geq 1$, there exists the largest ideal $J_n$ of $R$ such that $\si^i(J_n)$ is an ideal of $R$, for $0\leq i\leq n$. Notice that, by the construction, $J_{n+1}\subseteq J_n$, for every $n\geq 1$. Therefore $\si^n(\bigcap_{i=1}^\infty J_i)=\si^n(\bigcap_{i=n}^\infty J_i)=\bigcap_{i=n}^\infty \si^n(J_i)$ is an ideal of $R$, for all $n\geq 1$ and Proposition \ref{prop imagies of ideals} and Theorem \ref{thm automorphism} give the following:
\begin{corollary}
  Suppose that $\si$ is an endomorphism of a prime left noetherian ring such that $\si(R)$ contains a nonzero ideal of $R$. Let  $J_n$, where $n\in \mathbb{N}$, denote the largest ideal of $R$ such that $\si^i(J_n)$ is an ideal of $R$, for any $0\leq i\leq n$. Then all $J_n$'s are nonzero and $\si$ is an automorphism of $R$ iff $\bigcap_{i=1}^\infty J_i\ne 0$.
\end{corollary}

The following theorem records another situation  when every  endomorphism with a large image has to be an automorphism.

\begin{theorem}\label{PID}
 Suppose $\si$ is an endomorphism of a left principal
 ideal domain $R$ (left PID for short). If
 $\si(R)$ contains a nonzero left ideal $L$ of $R$,  then $\si$ is an automorphism of $R$.
\end{theorem}
\begin{proof} Suppose  $L$ is a nonzero left ideal of $R$ such that $L\subseteq \si(R)$. Hence,  by Theorem \ref{thm prime injective},  $\si$ is injective.

  Let $0\ne a\in L$.
 Then $ Ra\subseteq \si(R)$ and $ Ra$ is a principal left ideal of  $\si(R)$, as the ring $\si(R)$ is also a left PID. Thus there exists $c\in \si(R)$ such that $Ra=\si(R)c$.
 In particular $a=dc$, for some $d\in \si(R)$ and $Rdc=\si(R)c$.
 Since $R$ is a domain, $c$ is regular and we get $Rd=\si(R)$. Thus,
 as $1\in\si(R)$, we get $R=Rd=\si(R)$. This shows that $\si$ is onto.
\end{proof}

In the case $R$ is a left Ore domain we have the following:
\begin{prop} Let $\si$ be an injective endomorphism of a left Ore domain. If $\si(R)$ contains a nonzero one-sided ideal of $R$, then the extension of $\si$ to the division ring of quotients $D$ of $R$ is an automorphism of $D$.
\end{prop}
\begin{proof}
 Injectivity of $\si$ implies that $\si$ extends to an endomorphism of $D$. The assumption implies that there exists $0\ne c\in R$ such that either $cR$ or $Rc$ are contained in $\si(R)\subseteq \si(D)\subseteq D$. The fact that  $\si(D)$ is a division ring implies easily that $R\subseteq \si(D)$ and $\si(D)=D$.
\end{proof}

The following example presents a left PID $R$ with an injective
endomorphism $\si$ such that $\si(R)$ contains a nonzero right ideal
and $\si$ is not onto. Compare also this example with Theorem \ref{PID}.

\begin{example}
 Let $K$ be a field with an endomorphism $\si$ which is not onto.
 Consider the skew polynomial ring $R=K[x;\si]$ (with coefficients written on the left). We can extend $\si$ to $R$ by setting
 $\si(x)=x$. Then $R$ is left PID and $\si(R)x$ is a right ideal
 of $R$ contained in $\si(R)$.
\end{example}

In the corollary below we sum up obtained results in the special case of prime rings.
\begin{corollary}Suppose $R$ is a prime left noetherian ring. Let $\si$ be an endomorphism of  $R$ having a large image.
 Then:
 \begin{enumerate}
   \item  $\si$ is injective and extends to an automorphism of the classical left quotient ring of $R$;
   \item If $\si(L)=L$, for a certain essential left ideal of $R$, then $\si$ is automorphism of $R$;
   \item If $R$ is left PID, then $\si $ is an automorphism of $R$.
 \end{enumerate}
\end{corollary}
The above suggests the following:
\begin{Problem}
  Does there exist a prime left noetherian ring (or a left noetherian domain) with an endomorphism $\si$ such that $\si$ has a large image and $\si$ is not an automorphism of $R$.
  \end{Problem}
  The question seems to be interesting even in the case $R$ is a polynomial ring $K[X]$ over a field $K$ in the finite set $X=\{x_1,\ldots,x_n\}$ of indeterminates. Let $\tau \colon K[X]\rightarrow K[X]$ be a $K$-linear endomorphism. In the case $\# X=1$, Theorem \ref{PID} says that $\tau$ has to be an automorphism of $K[X]$. In general case, Proposition \ref{properties of large endomorhisms}(3)  shows that $K(X)$ is the quotient field of $K[\tau(X)]$.  It is known (Cf. Theorem 2.1 \cite{BCW}) that in this case the Jacobian Conjecture holds, i.e. $\tau$ has to be an automorphism, provided the Jacobian $J(\tau)\in K^*$.

     When char$K=0$, another special case of the question above  is: ``Does the Dixmier Conjecture, that every endomorphism of the   Weyl  algebra $A_1(K)$ is an automorphism, hold for endomorphisms having  large images''. Theorem  \ref{thm automorphism} implies  that this is exactly the case when there exists a nonzero left ideal  $L$ of $A_1(K)$, such that  $\si^n(L)$ is a left ideal of $A_1(K)$, for every $n\geq 1$.

We will present below examples of injective endomorphisms which are not automorphisms but have  large  images when $R$ is a prime left Goldie ring or even a commutative domain. The first one is a commutative local domain.
\begin{example}\label{local}
 Let $k(x)$ denote the field of rational functions over a field $k$ and
 $K$ its algebraic extension
 $K=k(x)(x^{\frac{1}{2^n}}\mid n\in \mathbb{N})$. Let $\si$ stand for the
  $k$-linear automorphism of $K$ given by
  $\si(x^{\frac{1}{2^n}})=x^{\frac{1}{2^{n-1}}}$. Then $\si$ can
  be extended to an automorphism of the power series ring
  $K[[y]]$, by setting $\si(y)=y$.
  Define $R=k(x)+K[[y]]y\subseteq K[[y]]$. Then $R$ is a local ring
  with the maximal ideal $M=K[[y]]y$, the restriction of $\si$ to $R$ is
   an injective endomorphism of $R$, which is not onto. Clearly $M=\si(M)$ is an
   ideal of $R$ contained in $\si(R)=k(x^2)+M\subset R$.
\end{example}
 The ring $R$ in the above example is not noetherian as otherwise, by Theorem \ref{thm automorphism}, $\si$ would be an automorphism of $R$. It is also easy to check directly that   if
  $I_n=(x^\frac{1}{2^i}y\mid 0\le i \le  n-1)$, $n\in \mathbb{N}$,
  then $x^\frac{1}{2^n}y\in I_{n+1}\setminus I_n$.

  In what follows, $U(R)$ will denote the unit group of $R$. Let us notice that $\si(U(R))\ne U(R)$ in Example \ref{local}. In fact let us observe that
  \begin{remark}\label{local ring automorphism on units}
   Let $R$ be local ring and $\si$ an injective endomorphism of
   $R$ such that $\si(U(R))=U(R)$. Then $\si$ is an automorphism
   of $R$. Indeed,  let $m\in M=R\setminus U(R)$. Since $R$ is local, $M$ is an ideal
of $R$. This implies that $1+m\not\in M$. Thus, by the assumption,
there is $a\in R$ such that $\si(a)=1+m$, i.e. $m=\si(a-1)\in
\si(R)$.
  \end{remark}

The following example offers a commutative domain $R$ with an
injective endomorphism $\si$ which is not onto  such that $\si(U(R))=U(R)$ and $\si(I)=I$,
for some nonzero ideal  $I$.
\begin{example}\label{ex. commuative}
Let $R=\mathbb{Z}+\mathbb{Z}x+\mathbb{Z}[\frac{1}{2}][x]x^2$. Let
$\si$ be the endomorphism of $R$ defined by $\si(x)=2x$. Then
 $I=\mathbb{Z}[\frac{1}{2}][x]x^2$
and $J=2\mathbb{Z}x+\mathbb{Z}[\frac{1}{2}][x]x^2$ are ideals of
$R$ contained in
$\si(R)=\mathbb{Z}+2\mathbb{Z}x+\mathbb{Z}[\frac{1}{2}][x]x^2$.
 Notice that  $\si(I)=I$ and $\si(J)\subset J$.
\end{example}

The ring $R$ in the above example is not noetherian. Indeed if
  $I_n=(\frac{1}{2^i}x^2\mid 0\le i \le  n-1)$, $n\in \mathbb{N}$,
  then $\frac{1}{2^n}x^2\in I_{n+1}\setminus I_n$.

Theorem \ref{prop. ext. to Q(R)} implies that if $R$ is a semiprime left Goldie ring with an injective endomorphism $\si$ having a large image, then its Cohn-Jordan extension $A=A(R,\si)$ is contained in $Q(R)$. Example \ref{ex. commuative} shows that   inclusions $R\subseteq A\subseteq Q(R)$  can be strict.

Example  \ref{ex. commuative}  can be generalized to the following
construction:
\begin{example}
 Let $T=\bigoplus_{n=0}^\infty T_n$ be an $\mathbb{N}$-graded ring
 with a graded automorphism $\si$, i.e. automorphism such that
 $\si(T_n)=T_n$. Let $R_0$ be a subring of $T_0$ such that $\si(R_0)\subset
 R_0$ and $R=R_0\oplus \bigoplus_{n=1}^\infty T_n$. Then $\si$ is
 an injective endomorphism of $R$ which is not an automorphism and
 $R$ contains an ideal $M=\bigoplus_{n=1}^\infty T_n$ of $R$ such
 that $\si(M)=M$.
\end{example}

It is easy to construct prime rings  or even domains, as in the above example. Take
any prime ring (or a domain) $R_0$  with an injective endomorphism $\si$ which is not onto. Let
$A_0=A(R_0,\si)$ be the Cohn-Jordan extension of $R_0$. Then $A_0$ is  a
prime ring (a domain) and consequently  $T=A_0[x]=\bigoplus_{n=0}^\infty T_n$ is also such a ring, where $T_n=A_0x^n$. Then one can extend $\si$ to and automorphism of $T$ and consider $R=R_0\oplus \bigoplus_{n=1}^\infty T_n$.

This construction never leads to noetherian rings. Notice that if $A_0$ would be
finitely generated as a left $R_0$-module, then $A_0=R_0$, i.e. $\si$ would be an
automorphism of $R_0$. Indeed, if $A_0=R_0a_1+\ldots +R_0a_m$, then there
would exist $n\geq 1$ such that $\si^n(a_i)\in R_0$, for $1\leq i\leq m$.
Then $A_0=\si^n(A_0)=\si^n(R_0a_1+\ldots +R_0a_m)\subseteq R_0$.

By the above, there are $a_i\in R_0$, $i\in\mathbb{N}$, such that
$R_0a_1+\ldots +R_0a_n\subset R_0a_1+\ldots +R_0a_{n+1}$, for all
$n$. Let $I_n$ denote  the left  ideal   of
$R=R_0+A_0[x]x$ generated by elements $a_1x,\ldots, a_nx$. Then
$I_n\subset I_{n+1}$, i.e. $R$ is not noetherian.


\begin{thebibliography}{99}
\bibitem{BCW}  Bass H., Connell E.H., Wright D., The Jacobian Conjecture:
Reduction of Degree and Formal Expansion of the Inverse, Bull. Amer. Math. Soc.  7 (2), 1982, 287-330;
%
%\bibitem{Be} Bergen J., A Note on the Primitivity of Ring Extensions,
%Comm. Algebra 23 (12), 1995, 4625-4631;

\bibitem{BFL}
   Beidar K.L.,  Fong Y.,    Lee P.-H.,  Wong T.-L.,  On Additive Maps of Prime Rings Satisfying the Engel Condition, Comm.  Algebra 25 (12), 1997,  3889-3902;

\bibitem{B}   Bresar M., One-Sided Ideals and Derivations of Prime Rings, Proc. Amer. Math. Soc. 122 (4), 1994, 979-983;

    \bibitem{B1}   Bresar M.,   Chebotar M.A.,  Martindale W.S., Functional Identities, Birkh\"{a}user, Basel, 2007;

\bibitem{BMM}  Bresar M., Martindale 3rd W.S.,  Miers  R.C., Maps Preserving $n$-th Powers, Comm. Algebra 26 (1),   1998,  117-138;

\bibitem{H} Hiremath V.A., Hopfian Rings and Hopfian Modules, Indian J. Pure and App. Math. 17 (7), 1986, 895-900;

\bibitem{Ja}
Jategaonkar, A.V.,
 Skew Polynomial Rings Over Orders In Artinian Rings,
 J. Algebra  21 (1), 51-59, 1972;

\bibitem{Jo}
 Jordan D.A., Bijective Extensions of Injective Rings
endomorphisms, J. London Math. Soc. 25 (3),   1982,   435-448;



\bibitem{LM} Leroy A.,  Matczuk J., D\'{e}rivations et automorphismes alg\'{e}briques
d'anneaux premiers, Comm. Algebra 13 (6),  1985, 1245-1266;

\bibitem{Ma} Matczuk J., S-Cohn-Jordan Extensions, Comm. Algebra 35 (3), 2007, 725-746;

\bibitem{MC} McConnell J.C.,      Robson J.C., Noncommutative Noetherian Rings, Graduate Studies in Mathematics  30, Amer. Math. Soc., 2001;

 \bibitem{T} Tripathi S.P., On the Hopficity of the Polynomial Rings, Proc. Indian Acad. Sci.    108, (2),  1998,   133-136;


 \bibitem{V1}   Varadarajan K., Study of Hopficity in Certain Classes of Rings, Comm. Algebra  28 (2), 2000, 771-783;

 \bibitem{V2}   Varadarajan K., Hopfian and co-Hopfian Objects, Publicacions Matem\`{a}tiques  36,  1992, 293-317;

\bibitem{V3}   Varadarajan K., Some Recent Results on Hopficity Co-Hopficity and Related Properties, International Symposium on Ring Theory, Trends in Math., Birkh\"{a}user, Boston, 2001.

\end{thebibliography}
\end{document}